\documentclass[11pt]{amsart}
\usepackage{graphicx}
\usepackage{amsmath}
\usepackage{amsfonts}
\usepackage{amssymb}
\usepackage{MnSymbol}
\usepackage{rotating}
\usepackage{amssymb}
\usepackage{gastex}
\usepackage[usenames]{color}

\newtheorem{thm}{Theorem}[section]
\newtheorem{cor}[thm]{Corollary}
\newtheorem{lem}[thm]{Lemma}
\newtheorem{prop}[thm]{Proposition}
\newtheorem{defn}[thm]{Definition}

\newtheorem{example}[thm]{Example}

\newcommand{\abs}[1]{\left\vert#1\right\vert}

\def\abs#1{\ensuremath{\lvert #1\rvert}}

\begin{document}
\title[Non-nilpotent graphs of semigroups]{The  Non-Nilpotent Graph of a Semigroup}
\author{E. Jespers and M.H. Shahzamanian}
\address{Department of Mathematics,
Vrije Universiteit Brussel,  Pleinlaan 2, 1050 Brussel, Belgium}
\email{efjesper@vub.ac.be, m.h.shahzamanian@vub.ac.be}
\thanks{ 
\\Research partially supported by
Onderzoeksraad of Vrije Universiteit Brussel, Fonds voor
Wetenschappelijk Onderzoek (Belgium). }

\begin{abstract}
We associate a graph ${\mathcal N}_{S}$ with a semigroup $S$
(called the upper non-nilpotent graph of $S$). The vertices of
this graph are the elements of $S$ and two  vertices  are adjacent
if they generate a semigroup that is not nilpotent (in the sense
of Malcev). In case $S$ is a group this graph has been introduced
by A. Abdollahi and M. Zarrin and some remarkable properties have
been proved. The aim of this paper is to study this graph (and
some related graphs, such as the non-commuting graph) and to
discover the algebraic structure of $S$ determined by the
associated graph. It is shown that if a finite semigroup $S$ has
empty upper non-nilpotent graph then $S$ is positively Engel. On
the other hand, a semigroup has a complete upper non-nilpotent
graph if and only if it is a completely simple semigroup that is a
band. One of the main results states that if all connected
${\mathcal N}_{S}$-components of a semigroup $S$ are complete
(with at least two elements) then  $S$ is a band that is a
semilattice of its connected components and, moreover, $S$ is an
iterated total ideal extension of its connected components. We
also show that some graphs, such as a cycle $C_{n}$ on $n$
vertices (with $n\geq 5$), are not  the upper non-nilpotent
graph of a semigroup.  Also, there is  precisely one graph on $4$
vertices that is not the upper non-nilpotent graph of a
semigroup with $4$ elements. This work also is a continuation of
earlier work by Okni\'{n}ski, Riley and the first named author on
(Malcev) nilpotent semigroups.
\end{abstract}

\maketitle

\section{Introduction}\label{pre}

Malcev \cite{malcev} and independently Neumann and Taylor
\cite{neu-tay} have shown that nilpotent groups can be defined by
using semigroup identities (that is, without using inverses). This
leads to the notion of a nilpotent semigroup (in the
sense of Malcev). It was shown in \cite{malcev,neu-tay} (see also
\cite{Okninski}) that a cancellative semigroup $S$ is nilpotent of
class $n$ if and only if $S$ has a group of fractions which is
nilpotent (in the classical sense) of class $n$. Properties of
nilpotent semigroups have been studied by Lallement in \cite{lal},
in particular he investigated the residual finiteness of finitely
generated nilpotent regular semigroups (extending Hall's result on
nilpotent groups). Meleshkin in \cite{mel} showed that free
nilpotent semigroups are cancellative and Grigorchuk \cite{gri}
showed that a finitely generated cancellative semigroup $S$ has
finite Gelfand-Kirillov dimension (or equivalently, its semigroup
algebra $K[S]$ over a field $K$ has finite Gelfand-Kirillov
dimension) if and only if $S$ is almost nilpotent (hence extending
a celebrated result of Gromov). So in particular, finitely
generated semigroup algebras of nilpotent semigroups have finite
Gelfand-Kirillov dimension.

In \cite{Eric} Jespers and Okninski studied the prime images of
semigroup algebras $K[S]$ of nilpotent semigroups $S$. It is shown
that there is a close relationship with prime images of group
algebras of nilpotent groups. The latter groups are closely
related to the image of $S$ in the prime images of $K[S]$. Further
the prime radical of $K[S]$ and the congruence it determines on
$S$ are described. Also a full description of nilpotent semigroups
of class $2$ is given. It turns out that one obtains a complete
analogue situation of the commutative case. In
\cite{jespers-okninski99} it is described when the contracted
semigroup algebra of a Malcev nilpotent semigroup is a prime
Noetherian maximal order.

Jespers and Riley in \cite{Riley} continued the investigations of
nilpotent semigroups within the class of linear semigroups, i.e.
subsemigroups of the multiplicative semigroup of all $n$-by-$n$
matrices over a field. For example, it is shown that the
nilpotence of a  linear semigroup can be characterized by a
$4$-generator semigroup condition, called the weakly Malcev
nilpotent conditions (WMN). This can be considered as some kind of
Engel's theorem for semigroups. Recall that Engel's famous theorem
in Lie theory gives a certain 2-generator criterion for the global
nilpotence of a finite-dimensional Lie algebra. It is also shown
that a finitely generated residually finite group is nilpotent if
and only if it is weakly Malcev nilpotent. Various other types of
local and global nilpotence conditions are studied in
\cite{Riley}, such as being positively Engel (PE) and Thue-Morse
(TM). In each case necessary and sufficient conditions are proved
for a linear semigroup to be of such a nilpotence type.

In the past twenty years fascinating questions and results have
been raised and investigated by studying graphs associated to
groups or rings (see for example
\cite{abd-zar,abd,akbari2,anderson,beck,bertram,Paul
Erdos,segev,williams}). In \cite{abd-zar} the notion of a
non-nilpotent graph $N_{G}$ of a group $G$ is introduced. The
vertices of the graph are the elements of $G$ and there is an edge
between vertices if they do not generate a nilpotent group. The
authors studied the group $G$ by the information that is stored in
this graph. Note that if the graph is empty (i.e. there are no
edges) then every two-generated subgroup is nilpotent. In
this case, if the group $G$ also is finite, then it is well known
that $G$ is a nilpotent group. One of the results proved in
\cite{abd-zar}  is that the number of connected components of
$N_G$ for a finite group $G$ is either $\abs{Z^{*}(G)}$ or
$\abs{Z^{*}(G)}+1$, where $Z^{*}(G)$ denotes the hypercenter of
$G$. Note the elements of $Z^{*}(G)$ are isolated points in the
graph $N_{G}$. Hence, in \cite{abd-zar},  one also studied the
induced subgraph $\mathfrak{N}_G$ on $G\backslash \mbox{nil} (G)$,
where $\mbox{nil} (G)$ is the subset of those elements $g\in G$
such that the group generated by $g$ and $h$ is nilpotent for any
$h\in H$. In general it is unknown whether $\mbox{nil} (G)$ is a
subgroup of G, but in many important cases it is. For example,
$\mbox{nil} (G)=Z^{*}(G)$ if $G$ is a finite group (or more
general, if
 G satisfies the maximal condition on its
subgroups) or if $G$ is a finitely generated solvable group.

The aim of this paper is  to  study graphs associated  to a
semigroup and to discover the algebraic structure of the semigroup
determined by its associated graph. The graphs of interest are
those determined by the (non) nilpotence of two-generated
subsemigroups. We mainly use notations as in \cite{cliford}.
Before stating our contributions, we first recall some
definitions.

For a semigroup $S$ with elements $x,y,z_{1},z_{2}, \ldots$
one recursively defines two sequences
$$\lambda_n=\lambda_{n}(x,y,z_{1},\cdots, z_{n})\quad{\rm and}
\quad \rho_n=\rho_{n}(x,y,z_{1},\cdots, z_{n})$$ by
$$\lambda_{0}=x, \quad \rho_{0}=y$$ and
$$\lambda_{n+1}=\lambda_{n} z_{n+1} \rho_{n}, \quad \rho_{n+1}
=\rho_{n} z_{n+1} \lambda_{n}.$$ A semigroup is said to be
\textit{nilpotent} (in the sense of Malcev \cite{malcev}) if there
exists a positive integer $n$ such that
$$\lambda_{n}(a,b,c_{1},\cdots, c_{n}) = \rho_{n}(a,b,c_{1},
\cdots, c_{n})$$ for all $a,b$ in $S$ and $ c_{1}, \cdots , c_{n}$
in $S^{1}$. The smallest such $n$ is called the nilpotency class
of $S$. Clearly, null semigroups are nilpotent. As mentioned
before (see for example \cite{Okninski}), a group $G$ is nilpotent
of  class $n$ if and only if it is nilpotent of class $n$ in the
classical sense. In \cite{Eric} it is proved that a completely
($0$-)simple  semigroup $S$  is nilpotent if and only if $S$ is an
inverse semigroup with nilpotent maximal subgroup $G$. Recall
(\cite{neu-tay}) that a semigroup $S$ is said to be
\textit{Neumann-Taylor} (NT)  if, for some $n\geq 2$,
$$\lambda_n(a,b,1,c_2,\cdots,c_n)=\rho_n(a,b,1,c_2,\cdots,c_n)$$
for all $a,b \in S$ and $c_2,\cdots,c_n$ in $S^1$. A semigroup $S$
is said to be \textit{positively Engel} (PE) if, for  some $n\geq
2$,
$$\lambda_{n}(a,b,1,1,c,c^{2},\cdots ,c^{n-2})
=\rho_{n} (a,b,1,1,c,c^{2},\cdots ,c^{n-2})$$ for all $a,b$ in $S$
and $c\in S^{1}$.

Recall (\cite{cliford}) that a semigroup $S$ is a completely
$0$-simple semigroup if and only if it is isomorphic with a
regular Rees matrix semigroup over a group  with zero, say $G^0$.
The group $G$ is a maximal subgroup of $S$. The standard notation
for such a semigroup $S$ is $\mathcal{M}^0(G,I,\Lambda;P)$, where
$I$ and $\Lambda$ are non-empty sets and $P$ is an $\Lambda \times
I$ matrix with entries in $G^0 = G \cup \{\theta\}$ (the latter is
the group $G$ adjoined with a zero element$~\theta$). The elements
of $\mathcal{M}^0(G,I,\Lambda;P)$ will be denoted as $g_{ij}$,
where $g \in G^0$, $i \in I$ and $j \in \Lambda$. Note that all
elements$~\theta_{ij}$, with $i\in I,\;  j \in \Lambda$, are
identified with the zero element of
$\mathcal{M}^0(G,I,\Lambda;P)$,  also denoted by $\theta$.  If $P$
contains no zero entry then $\mathcal{M}^0(G,I,\Lambda;P)
\backslash \{\theta\}$ is a completely simple semigroup which is
denoted as $\mathcal{M}(G,I,\Lambda;P)$.

As in \cite{Riley}, we denote by $F_{7}$ the semigroup which is
the disjoint union of the completely $0$-simple semigroup
$\mathcal{M}^0(\{e\},2,2;I_{2})$ and the cyclic group $\{1,u\}$
of order $2$:
\begin{eqnarray} \label{deff7}
F_{7} &=& \mathcal{M}^0(\{e\},2,2;I_{2}) \cup \{1,u\}
,\end{eqnarray}
 where $I_{2}$ denotes the
identity $2$-by-$2$ matrix.  The multiplication on $F_{7}$ is
defined by extending that of the defining subsemigroups via
$1s=s1=s$ for all $s \in \mathcal{M}^0(\{e\},2,2;I_{2})$, and
$e_{11}u=ue_{22}=e_{12}$, $e_{22}u=ue_{11}=e_{21}$. Note that
$$F_{7}=\langle u,e_{11}\rangle .$$ In
\cite{Riley} it is proved that a finite semigroup $S$  is
positively Engel if and only if all non-null  principal factors of
$S$ are inverse semigroups whose maximal subgroups are nilpotent
groups and $S$ does not have an epimorphic image that has $F_{7}$
as a subsemigroup.

Throughout the paper we will make frequently use of the above
mentioned results, without specific reference.

In order to study (local) nilpotence or commutativity of
semigroups, we define three types of graphs on a semigroup: the
upper non-nilpotent graph ${\mathcal N}_S$, the lower
non-nilpotent graph ${\mathcal L}_S$, and the non-commuting  graph
${\mathcal M}_S$. In general, these graphs are  different,
however, we show that if any of these graphs is complete then so
are the others. Moreover, in this case, it turns out that the
semigroup $S$ is completely simple. We investigate which graphs
can not show up as an upper non-nilpotent graph ${\mathcal N}_{S}$
for some semigroup $S$. In case $\abs{S}< 5$ then there is
only one such graph. Our main results focus on the two extreme
cases: (1) ${\mathcal N}_{S}$ is empty, that is every
two-generated subsemigroup is nilpotent, and (2) the connected
components are complete. The main results are the following:
\begin{enumerate}
\item If  $S$ is a semigroup, then ${\mathcal N}_{S}$ is complete
if and only if $S$ is a completely simple semigroup that is a
band. Moreover,  ${\mathcal N}_S$ is complete if and only if
${\mathcal L}_S$ (or equivalently ${\mathcal M}_S$) is complete.
Furthermore, for a finite semigroup $S$ of prime order, the graph
${\mathcal L}_{S}$ is connected if and only if it is complete.

\item If $S$ is a semigroup such that  all connected components of
${\mathcal N}_{S}$ are complete and have at least two elements
then $S$ is a band and  a semilattice of its connected components,
in particular $S$ is semisimple. Moreover, if $S_{i}$ and $S_{j}$
are distinct connected components, then $S_{i}\cup S_{j}$ is a
trivial total  ideal extension of $S_{i}$ by $S_{j}$, or
vice-versa, or $\abs{S_{i}S_{j}}=1$.

\item If $S$ is a finite semigroup such that $N_{S}$ is the empty
graph, then $S$ is positively Engel.

\item If $X$ is a graph with $4$ vertices, then $X$ is the upper
non-nilpotent graph of a semigroup if and only if $X\neq P_{4}$ (a
path on 4 vertices).

\item If $n\geq 5$ then $C_n$ (a graph which is a cycle on $n$
vertices) is not the upper non-nilpotent graph of any
semigroup.
\end{enumerate}


\section{Non-nilpotent graphs}\label{non-nil}

We begin by defining the graphs of interest on a semigroup $S$.

\begin{defn} \label{non-nilpotent-graph}
Let $S$ be a semigroup. The \textit{upper non-nilpotent graph}
${\mathcal N}_S$ of a semigroup $S$ is the graph whose vertices
are the elements of $S$ and in which there is an edge between two
distinct vertices $x$ and $y$ if and only if the subsemigroup
$\langle x, y\rangle$ generated by $x$ and $y$ is not a nilpotent
semigroup.
\end{defn}

The following lemma gives a criterion for a finite semigroup not
to be nilpotent.

\begin{lem} \label{finite-nilpotent}
A finite semigroup S is not nilpotent if and only if there exists
a positive integer $m$,  elements $x, y\in S$ and $w_{1}, w_{2},
\cdots, w_{m}\in S^{1}$ such that $x = \lambda_{m}(x, y,
w_{1}, w_{2}$ $, \cdots, w_{m})$, $y = \rho_{m}(x, y, w_{1},
w_{2}, \cdots, w_{m})$ and $x\neq y$ (note that for the converse
one does not need that $S$ is finite).
\end{lem}

\begin{proof}
Let $k=\abs{S}$. If $S$ is not nilpotent then there exist elements
$a, b\in S$ and  some $w_{1}, \cdots, w_{k^2+1}$ $ \in S^1$ such
that
$$\lambda_{k^2+1}(a,b,w_{1},\cdots, w_{k^2+1})
\neq \rho_{k^2+1}(a,b,w_{1}, \cdots, w_{k^2+1}).$$ Since
$\abs{S}^2 =k^2$ there exist positive integers $t$ and $r
\leq k^2+1$, $t< r$ with
 \begin{eqnarray*}
 \lefteqn{
(\lambda_{t}(a,b,w_{1},\cdots, w_{t}),
 \rho_{t}(a,b,w_{1},\cdots, w_{t})) } \\&=&
(\lambda_{r}(a,b,w_{1}, \cdots, w_{r}), \rho_{r}(a,b,w_{1},\cdots,
w_{r})).
 \end{eqnarray*}
Put $x= \lambda_{t}(a,b,w_{1}, \cdots, w_{t})$, $y=
\rho_{t}(a,b,w_{1},\cdots, w_{t})$ and $m=r-t$. Then $x =
\lambda_{m}(x, y, w_{t+1}, \cdots, w_{t+m}) \neq y = \rho_{m}(x,
y, w_{t+1}, \cdots, w_{t+m})$. This proves the necessity of the
stated condition. That this condition is sufficient is obvious.
\end{proof}

The lemma naturally leads us to the another graph on a
semigroup.

\begin{defn} The \textit{lower
non-nilpotent graph} ${\mathcal L}_{S}$ of a semigroup $S$ is the
graph whose vertices are the elements of $S$ and there is an edge
between two distinct vertices $x,y\in S$ if and only if there
exist  finitely many  elements  $w_{1}, w_{2}, \cdots, w_{n}$ in
$\langle x, y \rangle^{1}$ such that $x= \lambda_{n}(x, y, w_{1},
w_{2}, \cdots, w_{n})$ and $y= \rho_{n}(x, y, w_{1},$ $ w_{2},
\cdots, w_{n})$.
\end{defn}

Clearly ${\mathcal L}_{S}$ is a subgraph of ${\mathcal N}_{S}$. In
general these graphs are different. Indeed $F_{7}=\langle
u,e_{11}\rangle $ (see (\ref{deff7})) is not positively Engel, and
thus not nilpotent. So there is an edge between $e_{11}$ and $u$
and thus ${\mathcal N}_{F_{7}}$ is not empty. Since
$\mathcal{M}^0(\{e\},2,2;I_{2})$ and the cyclic group $\{1,u\}$
are nilpotent, there are no edges between elements of these
semigroups in their respective lower non-nilpotent graphs.
Further, because $\mathcal{M}^0(\{e\},2,2;I_{2})$ is an ideal of
$F_{7}$ it is impossible that for some positive integer $n$ one
has that $u= \lambda_{n}(u,x,w_1, \cdots, w_n)$ for some $x \in
\mathcal{M}^0(\{e\},2,2;I_{2})$ and some $w_1, \cdots, w_n \in
F_{7}$. Consequently, there is no edge in the graph ${\mathcal
L}_{F_{7}}$ between $u$ and any element of
$\mathcal{M}^0(\{e\},2,2;I_{2})$. Similarly, there is no edge
between $1$ and any element of $\mathcal{M}^0(\{e\},2,2;I_{2})$ in
${\mathcal L}_{F_{7}}$. Hence ${\mathcal L}_{F_{7}}$ is empty.

We now define a third graph; it contains ${\mathcal N}_{S}$ as a
subgraph.

\begin{defn}
The non-commuting  graph ${\mathcal M}_S$ of a semigroup  $S$ is
the graph whose vertices are the elements of $S$ and in which
there is an edge between two distinct vertices $x$ and $y$ if
these elements do not commute.
\end{defn}

\begin{lem} \label{empty-graph-TM}
Let $S$ be a semigroup. Assume ${\mathcal N}_S$ is an empty graph.
The following properties hold.
\begin{enumerate}
\item Assume $m$ is a positive integer, $x, y \in S$ and $w \in
S^{1}$. If $x = \lambda_{m}(x, y, w,$ $w, \cdots, w)$ and $y=
\rho_{m}(x, y, w, w, \cdots, w)$, then $x= y$. \item If an element
of $S$  has an inverse, then this inverse is unique.
\end{enumerate}
\end{lem}

\begin{proof} (1)
Assume $x$ and $y$ are different elements of $S$ and $w \in S^1$
such that $x = \lambda_{m}(x, y, w, \cdots, w)$ and $y=
\rho_{m}(x, y, w, \cdots, w)$ for some $m$. Then,  we have $xw =
\lambda_{m}(xw, yw, 1, 1, \cdots, 1)$ and $yw= \rho_{m}(xw, yw, 1,
1, \cdots, 1)$. As, by assumption  ${\mathcal N}_{S}$ is empty, we
get that $xw=yw$. Because of Lemma~\ref{finite-nilpotent}, we also
get that $\langle x,y,w\rangle$ is not a nilpotent semigroup.
Similarly, $wx= wy$. The equalities $wy=wx$, and $yw=xw$ and $y=
\rho_{m}(x, y, w, w, \cdots, w)$ imply that $y \in \langle x,
w\rangle$  and thus $\langle x, y, w\rangle = \langle x,
w\rangle$. Hence $\langle x, w\rangle$ is not a nilpotent
semigroup, a contradiction, because ${\mathcal N}_S$ is an empty
graph.

(2) Let $a\in S$. Suppose $b$ and $c$ are inverses of $a$ in $S$.
Then  it is easily verified that $b=bacacab$ and $c=cababac$.
Hence $b = \lambda_{2}(b, c, a, a)$ and $c = \rho_{2}(b, c, a,
a)$. So, by part (1),  $b = c$. Therefore $a$ has at most one
inverse, as desired.
\end{proof}

Clearly the lemma implies that every regular semigroup in which
every two-generated semigroup is nilpotent is an  inverse
semigroup.

\begin{thm} \label{empty-graph-Engel}
Let $S$ be a finite semigroup. If every two-generated
subsemigroup of $S$ is nilpotent (i.e. ${\mathcal N}_{S}$ is
empty) then $S$ is positively Engel.
\end{thm}

\begin{proof} Assume ${\mathcal N}_{S}$ is an
empty graph. Each principal factor of   $S$ is either completely
$0$-simple, completely simple or null. Every completely $0$-simple
factor is isomorphic with a regular Rees matrix semigroup $S'$
over a finite (maximal) subgroup $G$. Clearly, also ${\mathcal
N}_{S'}$ and ${\mathcal N}_{G}$ are empty graphs. Because of the
former, Lemma~\ref{empty-graph-TM} yields that $S'$ is an inverse
semigroup. The latter implies that every two-generated
subgroup of $G$ is nilpotent and thus the finite group $G$ is
nilpotent. Hence, $S'$ is a nilpotent semigroup  (see the
introduction). Therefore, every non-null principal factor of $S$
is an inverse semigroup with maximal subgroup a nilpotent group.

Furthermore, we know that the two-generated semigroup $F_{7}$
is not nilpotent. As ${\mathcal N}_{S}$ is empty, it therefore
follows that $F_{7}$ is not an epimorphic image of any
subsemigroup of $S$. Therefore, $S$  is positively Engel.
\end{proof}

It is well known that finite groups are positively Engel if and
only if they are nilpotent \cite{Riley}. However, in \cite{Riley},
it is shown that such a result is no longer true for finite
semigroups. Indeed an example is given of a finite semigroup that
is positively Engel but is  not nilpotent. We now give another
example: a finite semigroup with empty upper non-nilpotent graph
(and thus it is positively Engel)  but it  is not nilpotent.

Let $S$ be the semigroup that is the  disjoint union
$\mathcal{M}^0(\{e\}, 4, 4; I_4) \cup \{w, v\}$, where
$S'=\mathcal{M}^0(\{e\}, 4, 4; I_4)$ is a completely $0$-simple
subsemigroup of $S$ that is an ideal of $S$ and  such that the
following relations are satisfied: $w^2=v^2=wv=vw=\theta$,
$e_{11}w=e_{14}$, $e_{22}w=e_{23}$, $e_{33}w=\theta=e_{33}v$,
$e_{44} w=\theta=e_{44}v$, $e_{11}v=e_{13}$, $e_{22} v=e_{24}$,
$we_{11}=\theta=ve_{11}$, $we_{22}=\theta=ve_{22}$,
$we_{33}=e_{23}$, $we_{44}=e_{14}$, $ve_{33}=e_{13}$ and
$ve_{44}=e_{24}$.  We have $e_{31}=\lambda_2(e_{31},e_{42},w,v)$
and $e_{42}=\rho_2(e_{31},e_{42},w,v)$. Hence, by
Lemma~\ref{finite-nilpotent}, the semigroup  $S$ is not nilpotent.
We now show that ${\mathcal N}_{S}$ is an empty graph. Since the
semigroup $S'$ is nilpotent, there are no edges  between elements
of $S'$. Because the subsemigroup $ \langle w, v \rangle$ is
commutative, there is no edge between $w$ and $v$. Suppose now
that there is an edge between $s \in S'$ and $w$. Then, by
Lemma~\ref{finite-nilpotent}, there exist  elements $s_{1}, s_{2}
\in \langle w, s \rangle$ and some elements $w_1, w_2, \cdots, w_n
\in \langle w, s \rangle^1$  such that $s_{1}= \lambda_n(s_{1} ,
s_{2} , w_1, \cdots, w_n)$, $s_{2} = \rho_n(s_{1} , s_{2} , w_1,
\cdots,$ $w_n)$ and $s_1 \neq s_2$ (note that $s_{1}\neq \theta$
and $s_{2}\neq \theta$). Since $S'$ is an ideal it follows that
$s_{1},s_{2}\in S'$. As $S'$ is a nilpotent semigroup, we
furthermore obtain that at least one of the elements $w_{i}$
$(1\leq i\leq n)$ equals$~w$. Now, if necessary, replacing $s_{1}$
by $\lambda_{i-1} (s_{1},s_{2},w_{1},\cdots , w_{i-1})$ and
$s_{2}$ by $\rho_{i-1}(s_{1},s_{2},w_{1},\cdots , w_{i-1})$, we
may assume that $w=w_{1}$.  It then easily follows that $\theta
\not\in \{ s_{1} w, ws_{1} , s_{2} w, w s_{2} \}$ and thus $\{
s_{1}, s_{2} \} \subset \{e_{31},e_{32},e_{41},e_{42}\}$. Also,
$\theta\not\in \{ s_{1}ws_{2}, s_{2}ws_{1}\}$ and thus  $\{ s_{1}
, s_{2} \} = \{e_{31}, e_{42}\}$. Suppose $s_{1}=e_{31}$ and
$s_{2}=e_{42}$. Then $\lambda_{1}(s_{1},s_{2},w)= e_{32}$ and
$\rho_{1}(s_{1},s_{2},w)= e_{41}$. Thus, $\theta\neq
\lambda_{2}(s_{1},s_{2},w,w_{2}) =e_{32}w_{2}e_{41}$ and
$\theta\neq \rho_{2}(s_{1},s_{2},w,w_{2})=e_{41}we_{32}$. Hence,
$w_{2}\not\in S'$ and thus $w_{2}=w$. But then $\theta\neq
e_{32}w_{2}e_{41}=e_{32}we_{41}=e_{33}e_{41}=\theta$, a
contradiction. Similarly one shows that $s_{1}=e_{42}$ and
$s_{2}=e_{31}$ leads to a contradiction. Hence we have shown that
there are no edges between $w$ and any elements of $S'$.
Similarly, there are no edges between $v$ and elements of $S'$.
So, indeed, ${\mathcal N}_S$ is an empty graph.

One can improve the example a little in the sense that there
exists a finite semigroup $T$ with empty upper non-nilpotent graph
but $S$  is not Neumann-Taylor. One can take for $T$ the previous
example $S$ with an element $q$ added  such that $S$ is a
subsemigroup of $T$ and  an ideal of $T$ and, moreover, it
satisfies the relations $q^{2}=\theta$,
$e_{11}q=e_{33}q=qe_{22}=qe_{44}=\theta$, $e_{22}q = e_{21}$,
$e_{44}q=e_{43}$, $qe_{11}=e_{21}$, $qe_{33}=e_{43}$, $wq=e_{13}$,
$vq=e_{23}$, $qw=e_{24}$ and $qv=e_{23}$. We leave the
details to the reader.


\section{A description of semigroups with complete connected  components for the upper non-nilpotent graph}\label{conn}

In this section we give a description of semigroups for which all
connected components of the upper non-nilpotent graphs
are complete and contain at least two elements (so there are no
isolated vertices). To do so, we first describe the semigroups
that have complete upper non-nilpotent graphs.  It turns out
that this is equivalent with ${\mathcal L}_{S}$ (or ${\mathcal
M}_{S}$) being complete. We begin by showing that ${\mathcal
L}_{S}$ is a complete graph provided that ${\mathcal L}_{S}$ is
connected and $S$ is of prime order.

We start with a technical lemma.

\begin{lem} \label{ideal-lambda}
Let $S$ be a semigroup. The following properties hold.
\begin{enumerate}
\item If $I$ is an ideal of $S$ and  $x \in S \backslash  I$ then,
in the graph ${\mathcal L}_{S}$,  there is no edge between  $x$
and any element of $I$.

\item If an ideal $I$ of $S$ intersects non-trivially  a connected
component of ${\mathcal L}_S$ then this connected component is
contained in $I$.
\end{enumerate}
In particular, all vertices of a connected component of ${\mathcal
L}_{S}$ belong to the same $J$-class. Furthermore,  if ${\mathcal
L}_S$ is a connected graph, then $S$ is a simple semigroup.
\end{lem}

\begin{proof} (1) Suppose $I$ is an ideal of $S$ and
$x\in S\backslash I$.  Assume there is an edge in ${\mathcal
L}_{S}$  between  $y\in I$ and $x$. Then,  $x=
\lambda_{m}(x,y,z_{1}, \cdots , z_{m})$ and
$y=\rho_{m}(x,y,z_{1},\cdots , z_{m})$ for some $z_{1},\cdots
,z_{m} \in S^1$ and some positive integer $m$. Since $I$ is an
ideal of $S$, it is easily verified that then $x=
\lambda_{m}(x,y,z_{1}, \cdots , z_{m}) \in I$, a contradiction.

(2)  This follows easily from part (1).
\end{proof}

Recall that a \textit{complete graph}  is a graph in which every
pair of distinct vertices  are connected with an edge.

\begin{prop} \label{connected-complete} Let $S$ be a semigroup. The following properties hold.
\begin{enumerate}
\item If $S$ is a completely simple semigroup that is a band, i.e.
$S\cong {\mathcal{M}} (\{ e \}, I, \Lambda;P)$, for some sets $I$
and $\Lambda$ and sandwich matrix $P$ all whose entries are
$e$, then ${\mathcal L}_{S}$ is complete.

\item If $S$ has prime order and  ${\mathcal L}_S$ connected then
${\mathcal L}_{S}$ is complete.
\end{enumerate}
\end{prop}

\begin{proof} (1) Let $e_{ij}$ and $e_{kl}$ be elements of  $S= {\mathcal{M}} (\{ e \}, I, \Lambda;P)$. Then
$e_{ij}=\lambda_2(e_{ij},e_{kl},1,1)$ and
$e_{kl}=\rho_2(e_{ij},e_{kl},$ $1,1)$. Hence there is an edge
between $e_{ij}$ and $e_{kl}$ in ${\mathcal L}_{S}$. Therefore,
${\mathcal L}_S$ is complete.

(2) Assume $S$ has prime order $p$ and ${\mathcal L}_{S}$ is
connected. By Lemma~\ref{ideal-lambda}, $S$ is a simple semigroup.
Hence, $S$ is isomorphic with a regular Rees matrix semigroup over
a (maximal) subgroup $G$, that is $S = {\mathcal M}(G, I, \Lambda;
P)$ for some sets $I$ and $\Lambda$. Since $\abs{S}=p$, it follows
that  $S=G$ and $\abs{S}=\abs{G}=p$ or $G=\{ e\}$. The former is
excluded as a group of order $p$ is commutative and thus it has an
empty lower non-nilpotent graph (and thus not a complete graph).
Hence, $G = \{ e \}$  and ${\mathcal L}_{S}$ is complete by part
(1).
\end{proof}

Note that if ${\mathcal L}_{S}$ is not connected then a $J$-class
of $S$ may contain different connected components of ${\mathcal
L}_{S}$. For example the simple semigroup $S=
\mathcal{M}^0(\{1\},4,2;\left(
\begin{matrix} 1 & 1 & \theta & \theta\\ \theta & \theta & 1 & 1
\end{matrix} \right)
)$ has only one $J$-class, but has edges only between
$1_{11},1_{21}$ and $1_{32},1_{42}$ in ${\mathcal L}_{S}$.

We say that a vertex  $v$ of a graph is totally connected if there
are edges between $v$ and all other vertices of the graph.

\begin{lem}  \label{Complete-degree}
Let $S$ be a semigroup and ${\mathcal N}_S$ be its upper
non-nilpotent graph.  If a vertex $a$ in ${\mathcal N}_S$ is
totally connected, then $a$ is idempotent.
\end{lem}

\begin{proof}
If $a^2=b$, then $\langle a, b \rangle = \langle a, a^2 \rangle
=\langle a \rangle$. Thus, $\langle a, b \rangle$ is nilpotent and
there is no edge between $a$ and $b$. Which gives $a^2=a$.
\end{proof}

\begin{prop} \label{complete-graph}
The following conditions are equivalent for a semigroup$~S$.
\begin{enumerate}
    \item ${\mathcal L}_S$ is complete.
    \item ${\mathcal N}_S$ is complete.
    \item ${\mathcal M}_S$ is complete.
    \item $S$ is a completely simple semigroup that is a band, or equivalently,  $S =
    \mathcal{M}(\{e\},I,\Lambda;P)$.
\end{enumerate}
\end{prop}

\begin{proof}
The implications (1) implies (2) and (2) implies (3) are obvious.

To prove (3) implies (4), assume ${\mathcal M}_S$ is complete.
Then, since $x^{2}x=xx^{2}$ for $x\in S$, we get that  $x$ is
idempotent and thus $S$  is a band. Because ${\mathcal M}_S$ is
complete, we have that each idempotent is primitive. As $S$ is
band, we also obtain $a(aba)= (aba)a$ and $b(bab)=(bab)b$ for
elements $a, b \in S$. The completeness of ${\mathcal M}_S$
implies that $a=aba$ and $b=bab$. Hence any two elements of $S$
are inverses of each other. Therefore, $S$ is completely
simple that is a band (so its maximal subgroups are trivial).
Equivalently, $S=\mathcal{M}(\{ e \}, I, \Lambda; P)$ with $P$ a
sandwich matrix all whose components are equal to $e$. This proves
(4).

That (4) implies (1) follows from
Proposition~\ref{connected-complete}.
\end{proof}

\begin{lem} \label{idempotent-commute} Let $S$  be a band. The following properties hold.
\begin{enumerate}
\item ${\mathcal M}_S = {\mathcal N}_S$.

\item Each complete connected component of ${\mathcal N}_{S}$ is a
subsemigroup of $S$.

\item  If  ${\mathcal N}_{S}$  has no isolated vertex, then
each connected component of ${\mathcal N}_{S}$ is a subsemigroup
of $S$.
\end{enumerate}
\end{lem}

\begin{proof}
(1) Let $x$ and $y$ be two arbitrary distinct elements of $S$. If
$xy =yx$, then there is no edge between $x$ and $y$ in ${\mathcal
N}_S$ nor in ${\mathcal M}_{S}$. If $xy \neq yx$,  then $xy =
\lambda_{2}(xy, yx, 1, 1)$ and $yx = \rho_{2}(xy, yx, 1, 1)$,
because $x,\,  y,\, xy,\, yx$ are idempotent elements. Then, by
Lemma~\ref{finite-nilpotent} $\langle x, y \rangle$ is not
nilpotent. Hence there is an edge between $x$ and $y$ in both
${\mathcal N}_{S}$ and ${\mathcal M}_{S}$. Consequently,
${\mathcal M}_{S}={\mathcal N}_{S}$.

(2) Suppose that $a, b\in S$ are in the same complete connected
component of ${\mathcal N}_S$ but $c=ab$ is in a different
connected component.  As by the first part ${\mathcal M}_S =
{\mathcal N}_S$, we have that $c$ commutes with $a$ and $b$. We
then have $ba=(ba)(ba)= b(aba)= b(ca)= b(ac)=
b(aab)=b(ab)=bc=cb=abb=ab$. Therefore $ab=ba$, in contradiction
with the fact  there is an edge between $a$ and $b$ in ${\mathcal
M}_{S}={\mathcal N}_{S}$. Hence, indeed, each connected component
of ${\mathcal N}_S$ is  a subsemigroup.

(3) Suppose that $c, d\in S$ are in the same connected component
of ${\mathcal N}_S$  but $e=cd$ is in a different connected
component. As ${\mathcal N}_{S}$ has no isolated vertex, there
exists an element $f\in S$ such that there is an edge between $e$
and $f$ in ${\mathcal N}_{S}$. By the first part, $f$ commutes
with $c$ and $d$. Hence $fe= fcd = cfd = cdf = ef$. This yields a
contradiction with the fact that there is an edge between $e$ and
$f$ in ${\mathcal N}_{S}={\mathcal M}_{S}$.
\end{proof}

Note that if a band has an isolated vertex then in general its
connected components are not subsemigroups. For example, let
$B_{1}=\langle a,b\rangle$ and $B_{2}=\langle a,c\rangle$ be two
free bands and let $S$ be the semigroup that as a set is the union
$B_{1}\cup B_{2} \cup \{ \theta \}$  ($\theta \not\in B_{1}\cup
B_{2}$) and with multiplication such that $B_{1}$ and $B_{2}$ are
subsemigroups, $\theta $ is a zero element and
$xy=yx=\theta$ for $x\in B_{1}\backslash \langle a \rangle$, $y\in
B_{2}\backslash \langle a \rangle$. Then the connected components
of $S$ are $\{ \theta \}$ and $S\backslash \{ \theta \}$. Clearly
the latter is not a subsemigroup.

It is easy to give  an example of a finite semigroup $S$
such that ${\mathcal N}_{S}\neq {\mathcal M}_{S}$. Also  one
can easily construct a finite semigroup $T$ that is not a band and
for which  ${\mathcal N}_{T}={\mathcal M}_{T}$. In general, even
for a band $B$, one does not have that ${\mathcal L}_{B}={\mathcal
M}_{B}$.

In order to state the following result we recall (\cite{B
Steinberg}) that a total ideal extension of a semigroup $S$ by a
semigroup $T$ is a semigroup $M$ that is  the disjoint union $S
\cup T$ and that contains  $S$ and $T$ as subsemigroups  such that
$S$ is an ideal of $M$. If furthermore $st=s=ts$ for all $s\in S$
and $t \in T$, then we call $M$  a trivial total ideal extension
of $S$ by $T$ and we denote this by $T \angle S$. More generally,
if $n > 2$ then by  $S_1 \; \angle \; S_2 \; \angle \cdots \angle
\; S_n$ we denote the semigroup which is the disjoint  union
$\bigcup_{1 \leq i \leq n} S_i$ and for all $1 \leq i < j \leq n$,
$S_i \cup S_j = S_i \; \angle \; S_j$. Or equivalently, $S_1 \;
\angle \; S_2 \;  \angle \; \cdots \; \angle \; S_n$ is defined
recursively as $(S_1 \; \angle \; S_2 \; \angle \; \cdots \;
\angle \; S_{n-1})\; \angle \; S_n$.

\begin{lem}\label{total ideal}
Let $S_1, S_2, \cdots, S_n$ be semigroups. There is no edge
between any element of $S_i$ and any element of $S_j$ for $1\leq i
< j \leq n$ in ${\mathcal N}_{S_1 \; \angle \; S_2 \;  \angle \;
\cdots \; \angle \; S_n}$.

\end{lem}
\begin{proof}
If $s \in S_i$ and $t\in S_j$ for $1\leq i < j \leq n$, then
$st=ts$ because $S_i \cup S_j = S_i \; \angle \; S_j$. Hence there
is no edge between $s$ and $t$ in ${\mathcal N}_{S_1 \; \angle \;
S_2 \; \angle \; \cdots \; \angle \; S_n}$.
\end{proof}

We also recall another notion (see for example \cite{cliford}).
Suppose $S$ is a semigroup  such that $S= \bigcup \{S_{\alpha
}\mid \alpha \in \Omega \}$, a disjoint union of subsemigroups
$S_{\alpha}$, and  such that for every pair of elements $\alpha,
\beta \in \Omega$ we have $S_{\alpha}S_{\beta} \subseteq
S_{\gamma}$ for some $\gamma \in \Omega$. One then has a product
in $\Omega$ defined by $\alpha \beta = \gamma$ if
$S_{\alpha}S_{\beta} \subseteq S_{\gamma}$ and one says that $S$
is the union of the band $\Omega$ of semigroups $S_{\alpha}$, with
$\alpha \in \Omega$. If $\Omega$ is commutative, then one obtains
a partial order relation $\leq $ on  $\Omega$ with $\beta
\leq\alpha$ if $\alpha \beta =\beta$. In this case  $\Omega$ is a
semilattice and one says that $S$ is the semilattice $\Omega$
of semigroups $S_{\alpha}$.

\begin{thm}\label{semilattice-no-isolate}
Let $S$ be a semigroup and let $S_\omega$, with $\omega \in
\Omega$, denote the connected ${\mathcal N}_S$-components. The
following properties hold.
\begin{enumerate}
    \item If $\abs{S_{\omega}} > 1$, for each $\omega \in \Omega$, and if each connected ${\mathcal N}_{S}$-component is complete then
     $S$ is a band and each connected ${\mathcal N}_{S}$-component is a
    subsemigroup.
    \item If $S$ is a band and each connected ${\mathcal N}_{S}$-component is complete
    then ${\mathcal N}_{S}={\mathcal
    M}_{S}$,  the semigroup
    $S$ is a  semilattice $\Omega$ of the
    semigroups $S_{\alpha}$, and
    either $S_{\alpha} \cup S_{\beta}$ is a trivial total ideal extension of $S_{\alpha}$ by
    $S_{\beta}$ or of $S_{\beta}$ by $S_{\alpha}$ or
    $S_{\alpha}S_{\beta}= S_{\alpha \beta}$ with $\abs{S_{\alpha
    \beta}}=1$ for $\alpha, \beta \in \Omega$.
\end{enumerate}
\end{thm}
\begin{proof}
(1) Suppose  $\abs{S_{\omega}} > 1$ for each $\omega \in \Omega$.
Let $x_{1}\in S$ and put $x_{1}^{2}=x_{2}$. We need to prove that
$x_{1}=x_{2}$. Assume the contrary. Since $\langle
x_{1},x_{2} \rangle$ is commutative, the elements  $x_1$ and $x_2$
belong to different connected ${\mathcal N}_{S}$-components. Since
the connected component containing $x_{2}$ has more than one
element, there exists $x_3\in S$ such that there is an edge
between $x_2$ and $x_3$ in ${\mathcal N}_{S}$. Note that $x_{1}$
and $x_{3}$ are then in different connected ${\mathcal
N}_{S}$-components, and thus $\langle x_{1},x_{3}\rangle$  is
nilpotent. However, this yields a contradiction as $\langle x_2,
x_3 \rangle \subseteq \langle x_1, x_3 \rangle$ and $\langle
x_{2},x_{3}\rangle$ is not nilpotent.
Lemma~\ref{idempotent-commute} yields that each $S_{\alpha}$ is a
semigroup.

(2) Suppose $S$ is a band and each connected ${\mathcal
N}_{S}$-component is complete. By Lemma~\ref{idempotent-commute},
${\mathcal N}_{S}={\mathcal M}_{S}$ and each $S_{\alpha}$ is a
subsemigroup.

Fix $\alpha, \beta \in \omega$. Suppose $x_{1}\in S_{\alpha}$ and
$x_{2}\in S_{\beta}$, with $\alpha \neq \beta$.  So, $x_{1}$ and
$x_{2}$ commute. Put $x_{3}=x_{1}x_{2}$. Then $x_{3}$ commutes
with both $x_{1}$ and $x_{2}$.
 Let $\gamma \in \Omega$ be such that
$x_{3}\in S_{\gamma}$. We claim that if $x_{3}\not\in \{ x_{1}
,x_{2}\}$ then $\abs{S_{\gamma}} = 1$ and
$S_{\alpha}S_{\beta}=S_{\gamma}$. Indeed, suppose that
$x_{3}\not\in \{ x_{1} ,x_{2}\}$. Because, by assumption, all
connected ${\mathcal N}_{S}$-components are complete and
$x_{1}x_{3}=x_{3}x_{1}$, $x_{2}x_{3}=x_{3}x_{2}$, we get that
$x_{3}\not\in S_{\alpha} \cup S_{\beta}$. So $x_{3}\in S_{\gamma}$
with $\gamma \not \in \{ \alpha , \beta \}$. By
Proposition~\ref{complete-graph}, $S_{\gamma}=
\mathcal{M}(\{e\},I,\Lambda;P)$ with all entries of the sandwich
matrix $P$ equal to $e$.  Write $x_{3}=e_{ij}$ for some $i\in I$
and $j\in \Lambda$.

Assume $k\in I$ and  let $x_{4}=e_{kj}$.  Because $x_{1}$  and
$x_{4}$ belong to different connected ${\mathcal
N}_{S}$-components, we have that $x_{1}x_{4}=x_{4}x_{1}$.
Similarly, $x_{2}x_{4}=x_{4}x_{2}$. Since also
$x_{1}^{2}=x_{1}$ and $x_{4}^{2}=x_{4}$ we get that
$$x_{4}x_{3} = e_{kj} e_{ij}= x_{4},$$
$$x_{1} x_{4} = x_{4} x_{1} = e_{kj}x_{1} = e_{kj} e_{ij}x_{1}
=x_{4}
 x_{3}x_{1}=
  x_{4} x_{3} =e_{kj} e_{ij}= x_{4},$$
and
 $$x_{2} x_{4} =x_{4}x_{2} =x_{4}x_{3} x_{2}
 =x_{4} x_{3}= x_{4}.$$
Therefore,
 $$x_{3} x_{4}=x_{1}x_{2}x_{4} = x_{1}x_{4} = x_{4} .$$
So $x_{4}$ and $x_{3}$ commute and they belong to the same
complete connected ${\mathcal N}_{S}$-component. Therefore
$x_{4}=x_{3}$ and thus $k=i$ and thus $\abs{I}=1$. Similarly
$\abs{\Lambda}=1$ and thus $\abs{S_{\gamma}}=1$.

Assume $a \in S_{\alpha}$. Because $S_{\alpha}$ is completely
simple, we obtain from Proposition~\ref{complete-graph} that
$S_{\alpha}= \mathcal{M}(\{e'\},I',\Lambda';P')$ with all entries
of the sandwich matrix $P'$ equal to $e'$. Write $x_{1}=e'_{i'j'}
$ and $a= e'_{i''j''}$ for some $i', i''\in I$ and $j',
j''\in \Lambda$. Let $a'= e'_{i'j''}\in S_{\alpha}$. Then $$aa'
=a,\; a'a=a',\; a'x_1=x_1,\;  x_1a'=a'.$$ Note that we then also
get that
$$a'a=x_{1}a.$$ Since elements in different connected ${\mathcal
N}_{S}$-components commute we get that
$$x_1x_2 = a'x_1x_2 = x_1x_2a' = x_2x_1a' = x_2a' = a'x_2= a'ax_2 = x_{1}ax_{2} = x_1x_2a .$$ Thus we have proved that
 $ax_{1}x_{2} = x_{1}x_{2}a = x_{1}x_{2}$ for any  $a\in
S_{\alpha}$. Hence, for any $b\in S_{\alpha},$
 $$ ax_{1}bx_{2}=ax_{1}x_{2}b=x_{1}x_{2}b=bx_{1}x_{2}=x_{1}x_{2}. $$
As $S_{\alpha}$ is a completely simple semigroup, we
consequently obtain that $S_{\alpha}x_{2}=\{ x_{1}x_{2}\}$.
Because $S_{\alpha}$ and $S_{\beta}$ commute, we get by symmetry
that also $x_1S_{\beta} =S_{\gamma}$ and thus $S_{\alpha}S_{\beta}
= \{ x_1x_2\} =S_{\gamma }$. This proves the claim.

So, for the remainder of the proof, we may assume that
$x_{1}x_{2}=x_{2}x_{1}\in \{ x_{1} , x_{2}\}$  for all $x_{1}\in
S_{\alpha}$ and $x_{2}\in S_{\beta}$. Clearly we have that
$\langle S_{\alpha} \cup S_{\beta} \rangle = S_{\alpha} \cup
S_{\beta}$. We now show that either $S_{\alpha}$ or $S_{\beta}$ is
an ideal in $\langle S_{\alpha} \cup S_{\beta} \rangle $. For if
not, then, because of the symmetry in $S_{\alpha}$ and
$S_{\beta}$, we may assume that there exist $x_i \in S_{\alpha} $,
$x_j, x_{j'} \in S_{\beta}$ with $x_ix_j= x_i, x_ix_{j'}=x_{j'}$.
It follows that
$x_jx_{j'}=x_jx_ix_{j'}=x_ix_{j'}=x_{j'}x_i=x_{j'}x_ix_j=x_{j'}x_j$.
So $x_{j}$ and $x_{j'}$ distinct commuting elements in the
connected connected ${\mathcal N}_{S}$-component $S_{\beta}$, in
contradiction with the completeness of this component.
Consequently, $S_{\alpha} \cup S_{\beta}$ is a
trivial total ideal extension of $S_{\alpha}$ by $S_{\beta}$ or of
$S_{\beta}$ by $S_{\alpha}$.
\end{proof}

\begin{cor}\label{semilattice-no-isolate-finite}
Let $S$ be a finite semigroup and let $S_i$, for $1 \leq i \leq
n$, denote the connected ${\mathcal N}_S$-components. If the
connected component $S_{i}$ is complete and $\abs{S_i}> 1$
for all $1\leq i \leq n$, then $S= S_{\sigma (1)}\;  \angle \;
S_{\sigma (2)}\; \angle \; \cdots \; \angle \; S_{\sigma (n)}$ for
some $\sigma \in Sym_n$.
\end{cor}

\begin{proof}
Because of Theorem~\ref{semilattice-no-isolate}, we have $S_{i}
\cup S_{j} =S_{i}\; \angle \; S_{j}$ or $S_{i}\cup S_{j} =S_{j}\;
\angle \;  S_{i}$ for $1\leq i < j \leq n$. To prove the
result it thus is sufficient to show that if $S_{i_{1}} \cup
S_{i_{2}} = S_{i_{1}} \; \angle \; S_{i_{2}}$ and $S_{i_{2}} \cup
S_{i_{3}} = S_{i_{2}} \;  \angle \; S_{i_{3}}$ then $S_{i_{1}}
\cup S_{i_{3}} = S_{i_{1}} \; \angle \; S_{i_{3}}$. Suppose this
conclusion is false, i.e. assume that $S_{i_{1}} \cup S_{i_{3}} =
S_{i_{3}} \; \angle \; S_{i_{1}}$. Then, for $s_{1} \in
S_{i_{1}}$, $s_{2} \in S_{i_{2}}$ and $s_{3} \in S_{i_{3}}$  we
get that $(s_{1}s_{2})s_{3}=s_{2}s_{3}=s_{3}$ while
$s_{1}(s_{2}s_{3})=s_{1}s_{3}=s_{1}$, a contradiction.
\end{proof}

If $S$ has an isolated vertex (i.e. $S$ has a connected ${\mathcal
N}_{S}$-component with only one element)  in ${\mathcal N}_S$,
then the  theorem does not hold in general. For example, all
connected ${\mathcal N}_S$-components of the semigroup $S =
\{a,b,c,d\}$ defined by the following Cayley table are complete
but $S$ is not a band (the vertex $b$ is isolated and $\{a, c,
d\}$ is a complete subgraph in ${\mathcal N}_S$.)
$$
\begin{tabular}{ c|c c c c }
      & $a$ & $b$ & $c$ & $d$\\ \hline
  $a$ & $b$ & $a$ & $c$ & $d$ \\
  $b$ & $a$ & $b$ & $c$ & $d$ \\
  $c$ & $d$ & $c$ & $c$ & $d$ \\
  $d$ & $c$ & $d$ & $c$ & $d$ \\
\end{tabular}$$

We have shown in Lemma~\ref{idempotent-commute} that in a band $S$
all complete  connected ${\mathcal N}_{S}$-components    are
subsemigroups. In general this does not hold, for example if $S$
is an abelian semigroup then all connected ${\mathcal
N}_{S}$-components are isolated and of course not necessarily
subsemigroups. We do not know  whether the connected components of
the upper non-nilpotent graph of a semigroup are subsemigroups in
case there are no isolated vertices. In this context we have the
following remark.

\begin{prop}
Let $S$ be a semigroup. If ${\mathcal N}_S$ does not have isolated
vertices then, for every $x \in S$, the elements of the cyclic
subsemigroup $\langle x \rangle$ are all in the same connected
component.
\end{prop}

\begin{proof}
Suppose  $ x^n\in \langle x \rangle$ with $x^{n}\neq x$. Because
$x$ and $x^n$  commute, there is no edge between them in
${\mathcal N}_{S}$. As, by assumption, $S$ does not have isolated
vertices,  there exists an element $y\in S$ such that there is an
edge between $x^n$ and $y$ in ${\mathcal N}_{S}$. Clearly,
$\langle x^n, y \rangle \subseteq \langle x, y \rangle$. Thus,
$\langle x, y \rangle$ is not nilpotent and thus there is an edge
between $x$ and $y$. Therefore $x$ and $x^n$ are in the same
connected ${\mathcal N}_{S}$-component.
\end{proof}


\section{Some graphs are not upper non-nilpotent graphs}\label{}

For a positive integer $n$, we denote by  $P_n$ a graph which is a
path on $n$ vertices and by $C_n$ we denote a graph which is a
unique cycle on $n$ vertices. A star graph $S_n$, is a tree with
$n$ vertices such that one vertex (called the center) has degree
$n-1$ and the other $n-1$ vertices (called the terminal vertices)
have degree $1$.

Let $S$ be a semigroup. In  analogy with the group case \cite{abd-zar},  the
set of vertices of the non-nilpotent graph of $S$ which are not
adjacent to the vertex $x$, together with $x$, we call the
\textit{nilpotentizer} of $x$ in $S$. It will be denoted by
$\mbox{nil}_{S}(x)$. So
$$\mbox{nil}_{S}(x) = \{y \in S \mid \langle x, y
\rangle \mbox{ is nilpotent} \} .$$ The \textit{nilpotentizer} of
$S$ is the set
$$\mbox{nil}(S) = \bigcap_{x\in S}
\mbox{nil}_{S}(x) .$$ Thus $\mbox{nil}(S) = \{x \in S \mid \langle
x, y \rangle$ is nilpotent for all $y \in S\}$. Clearly, the
center $Z(S)$ of $S$ is contained in $\mbox{nil} (S)$. The
following lemma is easily verified.

\begin{lem}
If $y \in \mbox{nil}_{S}(x)$ and $z \in Z(S)$, then $yz \in
\mbox{nil}_{S}(x)$.
\end{lem}
Again in analogy with the group case \cite{abd-zar}, we call a
semigroup $S$ an \textit{$n$-semigroup} if $\mbox{nil}_{S}(x)$ is
a subsemigroup of $S$ for every $x \in S$. In general,
$\mbox{nil}_S(x)$ is not a subsemigroup of $S$ for all $x \in S$.
For example, $\mbox{nil}_S(1_{11})$ in the Rees matrix semigroup
$\mathcal{M}^0(\{1\},2,3;\left(
\begin{matrix} 1 & \theta\\ \theta & 1\\ 1 & \theta
\end{matrix} \right)
)$ is not a subsemigroup. Indeed the  subsemigroups
$\langle 1_{11},$ $ 1_{12} \rangle$ and $\langle 1_{11}, 1_{23}
\rangle$ are nilpotent, but $\langle 1_{11}, 1_{12}1_{23} \rangle$
is not nilpotent.

If $X$ is a graph on at most three vertices then one can
easily verify that $X$ is the upper non-nilpotent graph of a
semigroup. If $\abs{X} = 4$ then this no longer holds, we show now
that $P_4$ is the only exception. Recall that there are 126
semigroups with four elements.

\begin{figure}
\begin{picture}(00,30)(45,40)
\gasset{AHnb=0,Nw=2,Nh=2,Nframe=n,Nfill=y}
\gasset{ExtNL=y,NLdist=1,NLangle= 60}

  \node(0)(60,65){$d$}
  \node(1)(30,65){$a$}
  \node(2)(30,40){$b$}
  \node(3)(60,40){$c$}

  \drawedge(0,1){}
  \drawedge(1,2){}
  \drawedge(2,3){}
\end{picture}
\caption{Graph $P_4$} \label{fig1}
\end{figure}
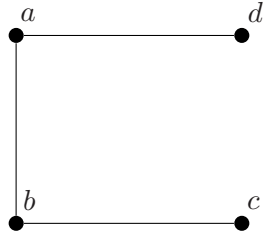

\begin{thm} \label{example-4-element}
The graph $P_{4}$ is not the upper non-nilpotent graph of a
semigroup with $4$ elements.
\end{thm}

\begin{proof} We prove the result by contradiction. So, suppose
$S = \{a, b, c, d \}$ is a semigroup  such that its upper
non-nilpotent graph is as depicted in Figure $1$. The proof is given in six steps. In the first five steps we prove several restrictions that we
may assume to hold in $S$. In step six we prove the final contradiction.

{\it Step 1:} $\langle c, d \rangle = \{c, d\}$, $\langle a, c
\rangle =\{a, c\}$, $\langle b, d \rangle =\{b, d\}$, $cd=dc$,
$ac=ca$, $bd=db$, $c^{2}=c$ and $d^{2}=d$.

Because $\langle c, d \rangle$ is nilpotent and $\langle a, d
\rangle$ is not nilpotent, it follows that $a \not\in \langle c, d
\rangle$. Similarly $b \not\in \langle c, d \rangle$. Hence
$\langle c, d \rangle = \{c, d\}$. We now show that this semigroup
is commutative. Indeed, for otherwise we may assume that $cd=c$
and $dc=d$. But $\lambda_{1}(c,d,1)=c$ and $\rho_{1}(c,d,1)=d$ and
thus, by Lemma~\ref{finite-nilpotent}, $\langle c, d \rangle$ is
not nilpotent, a contradiction. Similarly one obtains that
$\langle a, c \rangle = \{a, c\}$, $ac=ca$ and $\langle b, d
\rangle = \{b, d\}$, $bd = db$. Because $\langle a,d \rangle
\not\subseteq \langle a, c \rangle$, we have that $c^{2}\neq d$
and thus $c^{2}=c$. Similarly $d^{2}=d$.

{\it Step 2:} $ac=ca=c, cd=dc=c$, $bd=db=b$, $b \not\in \{ad,
da\}$ and $a \not\in \{bc, cb\}$.

We first note that $ac=c$ and $cd=d$ is not possible as it would
imply $ad=a(cd)=(ac)d=cd=dc=d(ca)=(dc)a=da$ and thus $\langle a, d
\rangle$ is abelian, a contradiction. Similarly, it is not
possible that $ac=a$ and $cd=c$. Thus we have that either $ac=c$
and $cd=c$, or $ac=a$ and $cd=d$. In the former case we obtain
$$c(ad) = (ca)d = (ac)d = cd = c = ac = a(cd) = a(dc) = (ad)c,$$
$$c(da) = (cd)a = ca= c = dc = d(ac) = (da)c$$ and in the latter
case one obtains $c(ad) = ad = (ad)c$, $c(da) = da = (da)c$. So,
in both cases, it follows that the elements $ad$ and $da$ commute
with $c$. As $\langle b, c \rangle$ is not nilpotent, it thus
follows that  $b \not\in \{ad, da\}$.

Similarly, we obtain that either $bd=b$ and $cd=c$, or $bd=d$ and
$cd=d$. In both cases we also have $a \not\in \{bc, cb\}$ (because
there is an edge between $b$ and $c$).

As a result of the above, we are left with two remaining cases:
 \begin{enumerate}
    \item $ac=ca=c, cd=dc=c$, $bd=db=b$, $b \not\in \{ad, da\}$ and $a \not\in \{bc,
cb\}$,
    \item $ac=ca=a, cd=dc=d$, $bd=db=d$, $b \not\in \{ad, da\}$ and $a \not\in \{bc,
cb\}$.
\end{enumerate}

Because of symmetry reasons we only have to deal with the first
case.

{\it Step 3:} $a \not\in \langle b,c \rangle$ and $b \not\in
\langle a, d \rangle$.

We prove this by
contradiction. So suppose $a \in \langle b, c \rangle$. Then there
 exists $\alpha \in S \backslash \{a\}$ such that $a=b\alpha$
or $a=c\alpha$. Both cases however are impossible as $\alpha \in
\{ b,c,d\}$ and $a\not\in \{bc,cb\}$, $bd=b$, $bb \in \langle b, d
\rangle$, $cd=c$ and $cc=c$. Hence the claim follows. Similarly,
we also get that $b \not\in \langle a, d \rangle$.

{\it Step 4:} $\langle a, d \rangle= \{a, d\}$, $ad\neq da$ and
$a^{2}=a$.

Since
$\langle a, d \rangle$ is not nilpotent,  we obtain that
there exist elements $w, v\in \langle a, d \rangle^{1}$ such that
$\lambda_{2}(a,d,w,v) \neq \rho_{2}(a,d,w,v)$. Because $b \not\in
\langle a, d \rangle$, we have that $\langle a,d\rangle \subseteq
\{ a,c,d\}$ and thus it follows that $\{awd, dwa\}$ must be one of
the following sets: $\{a,c\}$, $\{d,c\}$ or $\{a,d\}$. In the
first and second case, we get that  $c \in
\{\lambda_{1}(a,d,w,v),\; \rho_{1}(a,d,$ $w,v)\}$. As $\langle a,
d\rangle c =c\langle a,d \rangle =\{ c\}$, it then follows that
$\lambda_{2}(a,d,w,v) = \rho_{2}(a,d, $ $w,v) = c$, a
contradiction. Thus, we have $\{awd, dwa\}$ = $\{a, d\}$.

If $awd = a$ (and thus $dwa=d$) then $aw = a$, because $cd=c$ and
$dd=d$. Hence $ad = a$ and, clearly, $w\neq c$. If $w=1$ then also
$da = d$. If $w=a$, then also $a^2=a$ and $da = d$. If $w =d$,
then  it is also clear that $da=d$, because $d^2=d$. Hence,
we have shown that $\{ad, da\} = \{a, d\}$. Similarly,  one
can show that if $dwa = a$ then  $\{ad, da\} = \{a, d\}$.

Hence, we have proved that $\langle a, d \rangle = \{a, d\}$  and
thus $\langle a, d \rangle = \{ad, da\} = \{a, d\}$ (as $\langle
a, d \rangle$ is not nilpotent). It then follows that
$a=ad=a(da)=(ad)a=a^2$ or $a=da=(ad)a=a(da)=a^2$. Thus, we have
also  $a^2 =a$.

{\it Step 5:} $\langle b, c \rangle= \{b, c\}$, $bc\neq cb$ and
$b^{2}=b$.

Because $\langle b,c\rangle$ is not nilpotent there exist elements
$w, v\in \langle b, c \rangle^{1}$ such that $\lambda_{2}(b,c,w,v)
\neq \rho_{2}(b,c,w,v)$. Because $a \not\in \langle b, c \rangle$
we have that $\langle b,c\rangle \subseteq \{ b,c,d\}$ and thus it
follows that $\{bwc, cwb\}$ must be one of the sets: $\{b,d\}$,
$\{c,d\}$ or $\{b,c\}$.

If $\{bwc, cwb\} = \{b,d\}$ then $bw = b, wb = b$,  because
$\{dc,cd,c^2\} = \{c\}$. Hence $\{bc, cb\} = \{bwc, cwb\} = \{b,
d\}$. Therefore we have two cases.  If $bc = b$ and $cb=d$, then
$(cb)c= dc=c$ and $c(bc) = cb=d$, a contradiction. If $bc=d$ and
$cb=b$, then $(cb)c= bc=d$ and $c(bc) = cd=c$, a contradiction. So
the case $\{bwc,cwb\} = \{b,d\}$ is excluded.

If $\{bwc, cwb\} = \{c,d\}$ then, because $\{dc,cd,c^2\} = \{c\}$,
we get that $bw = b$ or $wb = b$. We claim that $d \in \{bc,
cb\}$. Indeed, if $bw=wb$, then $\{bc, cb\} = \{bwc, cwb\} =
\{c,d\}$. On the other hand, if $bw \neq wb$, then, since $\langle
bw,wb\rangle \langle b,c \rangle \subseteq \{ b,c,d\}$, we get
that $bw=b$ and $wb \in \{c, d\}$ and thus  $\{c,d\} = \{bwc,
cwb\} = \{bc, c\}$ and $d = bc$, or  $wb=b$ and $bw \in \{c, d\}$
and we obtain similarly that  $d = cb$. So this proves the claim.
Now, since $\{bwc, cwb\} = \{c,d\}$ then
$\{bwcvcwb, cwbvbwc\} =
\{cvd, dvc\}$.
If $v \in \{d,c, 1\}$ then
$$\{\lambda_{2}(b,c,w,v),
\rho_{2}(b,c,w,v)\} = \{bwcvcwb, cwbvbwc\} =
\{c\} ,$$ a contradiction. Thus $v=b$, and
$$\{\lambda_{2}(b,c,w,v),
\rho_{2}(b,c,w,v)\} = \{bwcvcwb, cwbvbwc\} =
\{cbd, dbc\} = \{cb, bc\} .$$ In particular $bc \neq cb$.  Since $d
\in \{bc, cb\} \subseteq \{b, c, d\}$, we also get that $b$ or $c$
belongs to $\{bc, cb\}$. If $b \in \{bc, cb\}$, then
$$\{bwcbcwb, cwb^3wc\} =\{bwcvcwb, cwbvbwc\} = \{bc, cb\} = \{b, d\}.$$
Thus $cwb^3wc \in \{b, d\}$ and therefore $cwb^3wc \neq c$. As
$c^2=c$ and $cd=c$, we obtain that $wb^3wc=b$. Similarly because
$dc=c$ and $c^2=c$, we get $cwb^3w=b$. Then $cwb^3wc=bc=cb$, a
contradiction with $\langle b, c \rangle$  not being  nilpotent.
Otherwise if $c \in \{bc, cb\}$ then $\{ bc,cb\} =\{ c,d\}$ and is
then readily verified that $(bc)b\neq b(cb)$, a  contradiction.
 So also the case $\{bwc, cwb\} = \{c,d\}$ is excluded.

Finally, we deal with the remaining case $\{bwc, cwb\} = \{b,c\}$.
We claim that $\{cb, bc\} = \{c, b\}$. Indeed recall that $w \in
\langle b, c \rangle ^1$ and $a \not\in \langle b, c \rangle$. If
$w\in \{c,d, 1\}$, then we get $\{b, c\} = \{bwc, cwb\} =\{bc,
cb\}$. If on the other hand $w=b$, then $b \in \{cb^2, b^2c\}$.
Hence, because $cd=dc=c^2=c$, we get $b^2=b$. So again $\{b,c\} =
\{bwc, cwb\}= \{cb, bc\}$. This proves the claim. Calculating
$(bc)b=b(cb)$, it follows that $b^2=b$.

{\it Step 6:} The final contradiction.

Of course $ab\in \{ a,b,c,d\}$. We show that each of the four
possible cases leads to a contradiction.

First assume $ab =a$. Then $ad \neq d$, because otherwise $b= db=
adb = ab=a$, a contradiction. As $\langle a, d \rangle = \{a, d\}$
and $ad \neq da$, we then get that $ad=a, da=d$. But then, $d=
da=dab=db=b$, again a contradiction.

Second assume $ab =b$. Then $ad \neq a$, because otherwise $ab =b=
bd= bda = ba$, in contradiction with $\langle a, b\rangle$ being
non-commutative. Since $\langle a, d \rangle = \{a, d\}$ and since
it is not nilpotent, we then get $ad=d, da=a$. Hence $b=
bd=b(ad)=(ba)d$. Because we also know that $ed=de=e$ for all $e
\in S \backslash \{a\}$, it therefore follows that $ba \neq c, ba
\neq d$. As $\langle b, a \rangle$ is non-commutative, we also get
that $ba\neq b$ (as $ab=b$). Thus $ba=a$. But then $d = ad =bad =
bd =b$, a contradiction.

Third, assume $ab=c$. Recall that $\langle a, d \rangle = \{ad,
da\} = \{a, d \}$. If $ad=d$, then $c = ab = adb = db = b$, a
contradiction. Hence $ad=a$ and $da=d$. Consequently, $c = dc=
dab=db=b$, again a contradiction.

Fourth, assume $ab=d$. Recall that $\langle b, c \rangle = \{bc,
cb\} = \{b, c \}$. If $cb=b$, then $d= ab= acb = cb=b$, a
contradiction. Hence $cb=c$ and $bc=b$. Consequently, $c = dc=abc
=  ab = d$, again a contradiction.

So we have reached in all possible cases a contradiction. Hence
the result follows.

\end{proof}

Note that the following Cayley table gives a semigroup $S$ with
$5$ elements such that its upper non-nilpotent graph has an
induced subgraph  on the set $\{b,c,d,e\}$  as in Figure $1$:
\begin{center}
\begin{tabular}{ c|c c c c c }
   & $a$ & $b$ & $c$ & $d$ & $e$\\ \hline
  $a$ &   $a$ & $a$ & $a$ & $a$ & $a$\\
  $b$ & $a$ & $b$ & $a$ & $a$ & $e$\\
  $c$ & $a$ & $a$ & $c$ & $d$ & $a$\\
  $d$ & $d$ & $d$ & $d$ & $d$ & $d$\\
  $e$ & $e$ & $e$ & $e$ & $e$ & $e$
\end{tabular}
\end{center}
Indeed  $\langle b, d \rangle$, $\langle d, e \rangle$ and
$\langle e, c \rangle$ are not nilpotent, but $\langle b, e
\rangle$, $\langle b, c \rangle$ and $\langle d, c \rangle$ are
nilpotent. Note however that $S \backslash \{a\}$ is not a
subsemigroup of $S$.

In order to show that $P_{4}$ is the only graph that does not show
as an upper non-nilpotent graph of a semigroup of order $4$, we
now first give several classes  of semigroups that are
$n$-semigroups. In each case we include a specific example. It
will follows that every graph on $4$ vertices, except
$P_{4}$, turns out to be the upper non-nilpotent graph of a
semigroup on $4$ elements that is of one of these types.

\begin{example} \label{nilp-examples}
{\rm The following classes of semigroups are all $n$-semigroups.

(1) {\it Semigroups with empty upper non-nilpotent graph.}

If ${\mathcal N}_S$ is an empty graph, then $\mbox{nil}_{S}(x) =
S$ for every $x \in S$. Of course examples of such semigroups
are commutative semigroups.

(2) {\it Semigroups with complete upper non-nilpotent graph.}

 If ${\mathcal N}_S$ is a complete graph then
$\mbox{nil}_{S}(x) = \{x\}$ for every $x \in S$. Because of
Proposition~\ref{complete-graph}, each element $x$ is idempotent
and thus $S$ is an $n$-semigroup. Furthermore, an example of such a semigroup
is a completely simple  band $\mathcal{M}(\{e\}, n, 1; P)$.

(3) {\it Semigroups whose upper non-nilpotent graph contains
only one pair of non-adjacent vertices.}

Suppose that $a$ and $b$ are the only elements of $S$ that are not
connected by an edge. If $x \in S \backslash \{a,b\}$, then the
vertex $x$ is totally connected and thus, by
Lemma~\ref{Complete-degree}, $x$ is idempotent. Hence
$\mbox{nil}_S(x)= \{x\}$ is a subsemigroup. Because $\langle a, b
\rangle$ is nilpotent and $\langle a, x \rangle$ is not nilpotent
for  $x \in S \backslash \{a,b\}$, we get that $x \not\in \langle
a, b \rangle$. Therefore $\langle a, b \rangle = \{a, b\}$ and
also $\mbox{nil}_{S}(a)=\mbox{nil}_S(b)=\langle a, b \rangle$.
Therefore $S$ is an $n$-semigroup. As an example one can take the
semigroup $T$ that is the disjoint union of the semigroup
$\mathcal{M}(\{e\}, n, 1; P)$ (with $n>1$) and the trivial group
$\{1\}$ and such that the following relations are satisfied:
$1x=e_{11}$ and $x1=x$ for every $x \in \mathcal{M}(\{e\}, n, 1;
P)$. Only between $1$ and $e_{11}$, there is no edge in
$\mathcal{N}_{T}$. Indeed, by Proposition~\ref{complete-graph},
${\mathcal N}_{\mathcal{M}(\{e\}, n, 1; P)}$ is a complete
subgraph. Since $1e_{11} = e_{11}1$, there is no edge between $1$
and $e_{11}$. If $ v \in \mathcal{M}(\{e\}, n, 1; P) \backslash
\{e_{11}\}$, then $\langle e_{11}, v \rangle \subseteq \langle 1,
v \rangle$. As $\langle e_{11}, v \rangle$ is not nilpotent, also
$\langle 1, v \rangle$ is  not nilpotent. Hence, there is an edge
between $v$ and $1$ in  ${\mathcal N}_{\mathcal{M}(\{e\}, n, 1; P)
\cup \{1\}}$.

(4) {\it Semigroups such that their upper non-nilpotent graph is a
disjoint union of a complete graph and one isolated vertex.}

Indeed, suppose that $x$ is an isolated vertex. Of course the
nilpotentizer of $x$ is $S$. If $y \in S \backslash \{x \}$ then
$\mbox{nil}_S(y)= \{x,y\}$. If $z \in S \backslash \{x, y\}$ then
$\langle y, z \rangle$ is not nilpotent and $\langle x, y \rangle$
is nilpotent. Hence $z \not\in \langle x, y \rangle$ and thus
$\langle x, y \rangle \cap S \backslash \{x, y\} = \emptyset$.
Therefore $\langle x, y \rangle = \{x, y\}$ and also
$\mbox{nil}_S(y)$ is a subsemigroup of $S$. Therefore, $S$ is an
$n$-semigroup. An example of such a semigroup is
$\mathcal{M}(\{e\}, n, 1; P)\; \angle \; S_1$ with $n>1$ and
$\abs{S_1}=1$.  By Proposition~\ref{complete-graph}, ${\mathcal
N}_{\mathcal{M}(\{e\}, n, 1; P)}$ is a complete subgraph and by
Lemma~\ref{total ideal}, there is no edge between
$\mathcal{M}(\{e\}, n, 1; P)$ and $S_1$ in ${\mathcal
N}_{\mathcal{M}(\{e\}, n, 1; P) \; \angle \; S_1}$.

(5) {\it Finite semigroups with complete connected upper
non-nilpotent components  such that each connected component
has more than one element.}

Because of Corollary~\ref{semilattice-no-isolate-finite} such a
semigroup is of the form  $S=S_{\sigma (1)} \; \angle \cdots \;
\angle \; $ $S_{\sigma (n)}$, where  $\sigma \in Sym_{n}$ and the
connected components are $S_{\sigma (i)}$, $1\leq i \leq n$.
Because of Theorem~\ref{semilattice-no-isolate}, $S$ is a band. If
$x\in S_{\sigma (i)}$, then clearly $\mbox{nil}_S(x)= \{x\} \cup
(S \backslash S_{\sigma (i)})$. It can be easily verified that
this is a subsemigroup of $S$. Therefore $S$ is an $n$-semigroup.

(6) {\it Semigroups such that their upper non-nilpotent graph is a
star graph.}

If $a$ is the center of the graph  ${\mathcal N}_S$ of such a
semigroup then, by Lemma~\ref{Complete-degree}, $a$ is idempotent.
Hence, $\mbox{nil}_S(a)= \{a \}$ is a subsemigroup. Between $a$
and any terminal element there is an edge. But between terminal
elements of $S$, say $c$ and $d$, there is no edge. Hence $cd\neq
a$, because otherwise $\langle a,c\rangle \subseteq \langle
c,d\rangle$. This yields a contradiction as the former
subsemigroup is not nilpotent while the latter is nilpotent.
Therefore $S \backslash \{a\}$ is a subsemigroup and $S$ is an
$n$-semigroup. An example of such a semigroup is $T_{n}=\{
x_{0},x_{1},\cdots , x_{n}\}$ ($n\geq 1$) with multiplication
defined by $x_{0}x_{i}=x_{0}$ and $x_{j}x_{i}=x_{1}$ for all $i$
and all $j\neq 0$.  Since $x_0x_1=x_0, x_1x_0=x_1$, we have
$x_0=\lambda_1(x_0,x_1,1)$ and $x_1=\rho_1(x_0,x_1,1)$
which implies the existence of an edge {in
${\mathcal{N}}_{T_{n}}$ between $x_{0}$ and $x_{1}$.} Moreover, since
$\langle x_1, x_0\rangle \subseteq \langle x_j,
x_0\rangle$ for $1 \leq j \leq n$ and $\langle x_1, x_0\rangle$ is
not nilpotent, $\langle x_j, x_0\rangle$ is not nilpotent and
there is an edge between $x_0$ and $x_j$ in ${\mathcal N}_{T_n}$.
As $x_jx_k=x_kx_j=x_1$ for $1 \leq j, k \leq n$ and $j \neq
k$, there is no edge between $x_k$ and $x_j$ in ${\mathcal
N}_{T_n}$. Therefore ${\mathcal N}_{T_n}$ is the star graph and
its center is the element $x_{0}$.

(7) {\it Semigroups such that their upper non-nilpotent graph
is a disjoint union of a star graph and one isolated
vertex.}

Let $a$ be the center of the star subgraph of such a semigroup $S$
and let $c$ be an isolated vertex. If $b\not\in \{ a,c\}$ then
$\langle a,b\rangle$ is not nilpotent while $\langle a,c \rangle$
is nilpotent. Hence $b\not\in \langle a,c \rangle$. Therefore,
$\mbox{nil}_S(a) = \{ a, c\} = \langle a, c \rangle$ is a
subsemigroup. It can also be easily verified  that the
nilpotentizer of  terminal elements of the star subgraph of $S$
are subsemigroups. Since also $\mbox{nil}_S(c)=S$ is a semigroup
we obtain that $S$ is an $n$-semigroup. An example of such a
semigroup  is $T_n \; \angle \; S_1$ with $\abs{S_1}=1$.

(8) {\it Semigroups such that their upper non-nilpotent graph
is $C_n$ with $n\leq 4$.}

In Theorem~\ref{C_n} it is shown that
$n\geq 5$ can not occur. If $n \leq 3$, then $C_n$ is a complete
graph and thus by (2) this semigroup is an $n$-semigroup. If $n=4$
then the statement is also easy to verify. The upper
non-nilpotent graphs of the semigroups with the following Cayley
tables
\begin{center}
 \begin{tabular}{ c|
c c c }
  & $a$ & $b$ & $c$\\ \hline
 $a$ &  $a$ & $a$ & $a$\\
  $b$ & $b$ & $b$ & $b$\\
 $c$ &  $c$ & $c$ & $c$
\end{tabular} \hspace{1cm} and \hspace{1cm}
\begin{tabular}{c| c c c c }
  & $a$ & $b$ & $c$ & $d$\\ \hline
  $a$   & $b$ & $b$ & $a$ & $b$\\
  $b$   & $b$ & $b$ & $b$ & $b$\\
  $c$   & $d$ & $d$ & $c$ & $d$\\
  $d$   & $d$ & $d$ & $d$ & $d$
\end{tabular}
\end{center}
are $C_3$ and $C_4$ respectively.}
\end{example}

\begin{figure}
\begin{picture}(00,30)(45,40)
\gasset{AHnb=0,Nw=2,Nh=2,Nframe=n,Nfill=y}
\gasset{ExtNL=y,NLdist=1,NLangle= 60}

  \node(0)(90,65){$b$}
  \node(1)(60,65){$a$}
  \node(2)(60,40){$c$}
  \node(3)(90,40){$d$}

  \node(4)(30,65){$b$}
  \node(5)(00,65){$a$}
  \node(6)(00,40){$c$}
  \node(7)(30,40){$d$}

  \drawedge(0,1){}
  \drawedge(0,3){}
  \drawedge(1,3){}
  \drawedge(1,2){}

  \drawedge(4,5){}

\end{picture}
\caption{} \label{fig2}
\end{figure}

\begin{prop} \label{main-converse}
Let $X$ be a graph with at most $4$ vertices. If $X\neq P_{4}$
then there exists a semigroup $S$ with $X$ as upper non-nilpotent
graph. Moreover, all such semigroups are $n$-semigroups.
\end{prop}

\begin{proof}
Let $X$ be a graph with at most $4$ vertices.  It can be
easily verified that if $X$ is not as in one of the graphs given
in Figure 2, then $X$ can be obtained as the non-nilpotent graph
of one of the semigroup types given in
Example~\ref{nilp-examples}. Furthermore, all semigroups with
upper non-nilpotent graphs of one of these types are
$n$-semigroups.

It can also  be easily verified that the  graphs in Figure 2 can
be obtained as the upper non-nilpotent graphs of the
following semigroups  with respective Cayley tables
\begin{center}
\begin{tabular}{c| c c c c }
  & $a$ & $b$ & $c$ & $d$\\ \hline
  $a$ & $a$ & $a$ & $a$ & $a$ \\
  $b$ & $a$ & $b$ & $a$ & $a$ \\
  $c$ & $a$ & $a$ & $c$ & $c$ \\
  $d$ & $a$ & $a$ & $d$ & $d$ \\
\end{tabular}
\hspace{1 cm} and \hspace{1 cm}
\begin{tabular}{c| c c c c }
  & $a$ & $b$ & $c$ & $d$\\ \hline
  $a$ & $a$ & $a$ & $a$ & $a$ \\
  $b$ & $b$ & $b$ & $b$ & $b$ \\
  $c$ & $b$ & $b$ & $c$ & $d$ \\
  $d$ & $d$ & $d$ & $d$ & $d$ \\
\end{tabular}
\end{center}
These semigroups are $n$-semigroups. So it remains to show that
all semigroups with upper non-nilpotent graphs as in Figure
2 are $n$-semigroups. Let $S=\{ a,b,c,d\}$  be such a
semigroup.

First we deal with the case  when  ${\mathcal N}_S$ is as the
graph depicted on the right in Figure $2$.  Since  $a$ is totally
connected, $a$ is idempotent by Lemma~\ref{Complete-degree}. Hence
$\mbox{nil}_S(a)= \{a\}$  is a subsemigroup. The subsemigroups
$\langle a, b \rangle$ and $\langle b, d \rangle$ are not
nilpotent, while $\langle b, c \rangle$ is nilpotent. Hence $a$
and $d$ are not in $\langle b, c \rangle$. Therefore,
$\mbox{nil}_S(b)= \{b,c\} =\langle b,c\rangle$ is a subsemigroup.
Because $\langle a, d \rangle$ and $\langle b, d \rangle$ are not
nilpotent, but $\langle c, d \rangle$ is nilpotent, we get that
$a$ and $b$ are not in $\langle c, d \rangle$. Hence $\langle c, d
\rangle = \{c, d\}$. Therefore $\mbox{nil}_S(d)$ is a
subsemigroup.

Since the order of both subsemigroups $\langle b, c \rangle$
and $\langle c, d \rangle$ is two and because  these semigroups
are nilpotent,  one can easily verify that  $cd=dc$ and
$bc=cb$. We claim that $bd\neq a$. Indeed, because otherwise we
have $ac = bdc = cbd = ca$, in contradiction with $\langle a,c
\rangle$ not being nilpotent. Similarly $db \neq a$. Hence $a
\not\in \{bd, db\}.$ As also $\langle b, c \rangle = \{b, c\}$ and
$\langle c, d \rangle = \{c, d\}$, we obtain  that
$\mbox{nil}_S(c) = \{c, b, d\}$ is a subsemigroup. Therefore $S$
is an $n$-semigroup.

Next we deal with the case  when  ${\mathcal N}_{S}$ is  the
graph depicted on the left in Figure $2$.  We claim $cd \neq a$.
We prove this by contraction.  Assume $cd=a$. Because there
is an edge between $a$ and $b$, but there is no edge between $b$
and $d$ and $cd=a$,  $a$ and $c$ are not in $\langle b, d
\rangle$. Hence $\langle b, d \rangle = \{b, d\}$  and since
$\langle b, d \rangle$ is nilpotent,  $bd=db$. Also, as there is
no edge between $a$ and $d$, we have that $b \not\in \langle a, d
\rangle$.  Since $\langle b, d \rangle = \{b, d\}$, we have
$d^2=d$. Similarly $\langle b, c \rangle = \{b, c\}$, $bc=cb$,
$c^2=c$. Because $c$ and $d$ are idempotent and $cd=a$, we get
$ad=a$ and $ca=a$.

Because there is no edge between $c$ and $d$ and $cd=a$, one has
$b\not\in \langle d, c \rangle$. Hence $dc \neq b$. If $dc=d$,
then $dcd=d$ and $da=d$.  Therefore $a=\lambda_1(a,d,1)$ and
$d=\rho_1(a,d,1)$. Hence, by Lemma~\ref{finite-nilpotent},
$\langle a, d \rangle$ is not nilpotent, a contradiction. So
$dc\neq d$ and, similarly, $dc \neq c$. Therefore $dc=a$. Because
$c$ and $d$ are idempotent and $dc=a$, we have $da=a$ and $ac=a$.

So we have $ac=ca=ad=da=cd=dc=a$, $bd=db$ and $bc=cb$.
Consequently, $dba=bda=ba$ and $cba=bca=ba$. If $ba=c$ then
$c= ba=dba=dc=a$,  a contradiction. If $ba=d$ then $d=
ba=cba=cd=a$, a contradiction too. Thus $ba\in \{a,b\}$. Similarly
$ab\in \{a,b\}$. As there is an edge between $a$ and $b$, we have
that $ab \neq ba$ and thus $\{ab, ba\}= \{a,b\}$.

Suppose $ab=a$ and  $ba=b$. Then $db=dba=bda=ba=b$ and
$cb=cba=bca=ba=b$ and thus $b=cb=c(db)=(cd)b=ab=a$, a
contradiction. Similarly, $ab=b$ and $ba=a$ lead to a
contradiction.

This proves the claim that $cd \neq a$.  Similarly $\{a,b\}
\cap \langle c, d \rangle = \emptyset$.  Consequently $\langle
c,d\rangle =\{ c, d\}$. It then can be  easily verified from the graph of
 $S$ that $\mbox{nil}_{S}(a)=\{a,c,d\}$ and
$\mbox{nil}_{S}(b)=\{b,c,d\}$ are subsemigroups. Therefore $S$ is
an $n$-semigroup.
\end{proof}

\begin{cor}
If $X$ is a graph with $4$ vertices then there exists a semigroup
$S$ with $4$ elements  such that $X = {\mathcal N}_S$ if and
only if $X \neq P_4$.
\end{cor}
\begin{proof}
This follows  immediately from
Theorem~\ref{example-4-element}, Example~\ref{nilp-examples} and
Proposition~\ref{main-converse}.
\end{proof}

Next we show that cycle graphs  with at least 5 vertices can
not be upper non-nilpotent graphs of semigroups.

\begin{thm} \label{C_n}
If $S$ is a finite semigroup of order $n \geq 5$ then ${\mathcal
N}_S \neq C_n$.
\end{thm}
\begin{proof}
Suppose that the semigroup $S = \{a_1, \cdots, a_n \}$ is such
that ${\mathcal N}_S = C_n$, and there are edges between $a_{i}$
and $a_{i+1}$ for $1 \leq i \leq n$ (the addition  used in the
indices has to be interpreted modulo $n$).

First we show that $S$ is a band. Indeed, suppose
$a_{i}^{2}=a_{j}$. Clearly $\langle a_{j+1},\; a_{j} \rangle =
\langle a_{j+1},\; a_{i}^2 \rangle \subseteq \langle a_{j+1},
a_{i} \rangle$ and $\langle a_{j-1},\; a_{j} \rangle = \langle
a_{j-1},\; a_{i}^2 \rangle \subseteq \langle a_{j-1}, a_{i}
\rangle$. Because $\langle a_{j+1},\; a_{j} \rangle$ and $\langle
a_{j-1},\; a_{j} \rangle$ are not nilpotent, we get that $\langle
a_{j+1}, a_{i} \rangle$ and $\langle a_{j-1}, a_{i} \rangle$ are
both  not nilpotent. Hence $j=i$ and thus $S$ is band, because $n
\geq 5$. Because of Lemma~\ref{idempotent-commute}, we then have
that ${\mathcal M}_S = {\mathcal N}_S$.

We claim that, $\langle a_{i}, a_{i+1} \rangle = \{a_{i},
a_{i+1}\}$. Since $n\geq 5$ there is no edge between $a_{{i-2}}$
and any of the elements $a_{i}$  and $a_{{i+1}}$. Hence,
$a_{i-2}$ commutes with each of these elements and thus also with
$a_{i}a_{{i+1}}$. It follows that $a_{i}a_{i+1} \not\in \{
a_{i-1}, a_{i-3} \}$.

If $n=5$ it then follows that  $\{a_{i+1}a_{i}, a_{i}a_{i+1}\}
\subseteq \{a_i, a_{i+1}, a_{i-2}\}$. If $a_{i}a_{i+1} = a_{i-2}$,
then $a_{i+1}a_{i} \in \{a_{i}, a_{{i+1}}\}$, because
$a_{i+1}a_{i} \neq a_{i}a_{i+1}$. We suppose $a_{i+1}a_{i} =
a_{i+1}$. Then we have
$$a_{i-2}a_{i}=(a_{i}a_{i+1})a_{i} = a_{i}(a_{i+1}a_{i}) =
a_{i}a_{i+1} = a_{i-2},$$
$$a_{i-2}a_{i+1}= a_{i+1}a_{i-2}=a_{i+1}(a_{i}a_{i+1}) = (a_{i+1}a_{i})a_{i+1} =
a_{i+1}a_{i+1} = a_{i+1}.$$ Then
$$a_{i+1} = a_{i-2}a_{i+1}= (a_{i-2}a_{i})a_{i+1} = a_{i-2}(a_{i}a_{i+1}) =
a_{i-2}a_{i-2} = a_{i-2},$$ a contradiction. Similarly
$a_{i+1}a_{i} = a_{i}$ yields a  contradiction. Therefore
$a_{i}a_{i+1} \in \langle a_{i}, a_{{i+1}} \rangle$. Similarly
$a_{i+1}a_{i} \in \langle a_{i}, a_{{i+1}} \rangle$. Then $\langle
a_{i}, a_{{i+1}} \rangle = \{a_{i}, a_{{i+1}}\}$.

Now we suppose that $n \geq 6$ and $a_{i}a_{i+1} = a_{j}$ such
that $j \not\in \{i, {i+1}\}$. Because $n \geq 6$, there exists
element $a^\star \in \{a_{j+1}, a_{j-1}\}$ such that $a^\star$ is
not adjacent to $a_{i}$ and to $a_{i+1}$. Which leads $a^\star
a_{i}a_{i+1}=a_{i}a_{i+1}a^\star$. Therefore $a_{j}a^\star=a^\star
a_{j}$, a contradiction. Then we have $\langle a_{i}, a_{i+1}
\rangle = \{a_{i}, a_{i+1}\}$.

Since the subsemigroup $\langle a_{i}, a_{i+1} \rangle$ is
not nilpotent, we thus obtain that $\{a_{i}a_{i+1}, a_{i+1}a_{i}\}
= \{a_{i},\; a_{i+1}\}$. We claim that if  $a_{i}a_{i+1}= a_{i}$
then $a_{i+1}a_{i+2}= a_{i+1}$. Indeed, for otherwise we  have
$$a_{i}a_{i+1}= a_{i}, a_{i+1}a_{i}= a_{i+1}, a_{i+1}a_{i+2}=
a_{i+2}, a_{i+2}a_{i+1}= a_{i+1}.$$ Note that the elements $a_{i}$
and $a_{i+2}$ commute. Hence, we get
$$a_{i} = a_{i}a_{i+1} = a_{i}a_{i+2}a_{i+1} =
a_{i+2}a_{i}a_{i+1} = a_{i+2}a_{i} =$$
$$a_{i+1}a_{i+2}a_{i} = a_{i+1}a_{i}a_{i+2} =
a_{i+1}a_{i+2} = a_{i+2},$$ a contradiction. This proves the
claim.

Suppose now that $a_{i}a_{i+1}=a_{i}$. Since $a_{i}a_{i+1}=a_{i}$,
we have  from the above that $a_{i+1}a_{i+2}=a_{i+1}$ and so
$a_{i}a_{i+2}=a_{i}a_{i+1}a_{i+2}=a_{i}a_{i+1}=a_{i}$. Because
$a_{i}a_{i+1}=a_{i}$, $a_{i+1}a_{i+2}=a_{i+1}$, $a_{i}a_{i+1} \neq
a_{i+1}a_{i}$, $a_{i+1}a_{i+2} \neq a_{i+2}a_{i+1}$, $\langle
a_{i}, a_{i+1} \rangle = \{a_{i}, a_{i+1} \}$ and $\langle
a_{i+1}, a_{i+2}\rangle = \{a_{i+1}, a_{i+2}\}$, we get that
$a_{i+1}a_{i}=a_{i+1}$ and $a_{i+2}a_{i+1}=a_{i+2}$. Hence,
$a_{i+2}a_{i}=a_{i+2}a_{i+1}a_{i}=a_{i+2}a_{i+1}=a_{i+2}$ and thus
$\langle a_{i}, a_{i+2} \rangle = \{a_{i}, a_{i+2}\}$.

Similarly if $a_{i}a_{i+1}=a_{i+1}$, we have $\langle a_{i},
a_{i+2} \rangle = \{a_{i}, a_{i+2}\}$.

The above information shows that $ \{a_i, a_{i+1}, a_{i+2}\}$ is a
subsemigroup.  Since $a_i = \lambda_1(a_i, a_{i+1}, 1), a_{i+1} =
\rho_1(a_i, a_{i+1}, 1)$ or $a_i = \lambda_2(a_i, a_{i+1}, 1, 1),
a_{i+1} = \rho_2(a_i, $ $a_{i+1}, 1, 1)$, there is an edge between
$a_i$ and $a_{i+1}$ in ${\mathcal L}_{\{a_i, a_{i+1}, a_{i+2}\}}$.
Similarly there is an edge between $a_{i+1}$ and $a_{i+2}$ in
${\mathcal L}_{\{a_i, a_{i+1}, a_{i+2}\}}$. But there is no edge
between $a_i$ and $a_{i+2}$ in this graph, because they
commute. As the the order of the semigroup $\{ a_{i},
a_{i+1},a_{i+2}\}$ is of prime order and the lower non-nilpotent
graph of this semigroup is connected but not complete this yields
a contradiction with Proposition~\ref{connected-complete}.
\end{proof}

Note that $P_4$ and $C_5$ are isomorphic to their respective
complements.  So a question of interest is whether a  graph that
is isomorphic to its complement graph can occur as an upper
non-nilpotent graph of a semigroup. Note that $C_n$ (expect $n=5$)
is not isomorphic to its complement graph.

\vspace{12pt}
{\bf Acknowledgment}
The authors would like to thank the referee for a thorough report that resulted in a  much improved version of the paper.


\end{document}